\documentclass[12pt]{amsart}
\usepackage{amssymb}
\usepackage{amsfonts}
\usepackage{amssymb,latexsym}
\usepackage{enumerate}
\usepackage{mathrsfs}

\usepackage{graphicx}
\usepackage[all]{xy}

\newtheorem{thm}{Theorem}[section]

\newtheorem{cor}[thm]{Corollary}

\newtheorem{prop}[thm]{Proposition}
\newtheorem{mprop}[thm]{Main Proposition}

\newtheorem{rem}[thm]{Remark}

\theoremstyle{definition}
\newtheorem{definition}[thm]{Definition}

\theoremstyle{remark}

\begin{document}

\title[{\normalsize {\large {\normalsize {\Large {\LARGE }}}}}  counter examples to the nonrestricted representation theory
]{ Counter examples to the nonrestricted representation theory}

\author{Kim YangGon }

\address{emeritus professor,
	 Department of Mathematics,  
	Jeonbuk National University, Korea.}

\

\

\email{ kyk1.chonbuk@hanmail.net }


\begin{abstract}
We shall consider nonrestricted representations of $C_l-$ type Lie algebra over an algebraically closed field of characteristic  $p\geq7.$
 This paper gives some counter examples to important theory relating to the representations of modular Lie algebras.

\end{abstract}

\maketitle

\section{introduction}

\

\

\large{We shall study the $C_l-$ type Lie algebra in section 2. In section 3, restricted and nonrestricted representations are explained.
Next we exhibit irreducible modules for nonrestricted representations of modular $C_l-$ type  Lie algebra in the last section 4.\newline

This paper just gives some counter examples to [3] which states sort of the dimensions of  irreducible modules over modular $C_l-$ type Lie algebras.

\

\section{ Modular $C_l-$ type Lie alebra }

\

\

As is well known the symplectic Lie algebra of rank $l$, i.e., the $C_l-$ type Lie algebra over $\mathbb{C}$  has its root system $\Phi $= ${\{ \pm2\epsilon_i, \pm(\epsilon_i
\pm\epsilon_j),i\neq j}\}$, where $\epsilon_i,$ $\epsilon_j$ are linearly independent unit vectors in  $\mathbb{R}^l$ with $l\geq3.$ \newline

For an algebraically closed field $F$ of prime characteristic $p,$ the $C_l-$ type Lie algebra $L$ over $F$ is just the analogue over $F$ of the $C_l-$ type 
simple Lie algebra  over $\mathbb{C.}$\newline

 In other words the $C_l-$ type Lie algebra over $F$ is isomorphic to the Chevalley Lie algebra of the form 
$\sum_{i=1}^{n}\mathbb{Z}c_i\otimes_\mathbb{Z}F,$ where $ n$= $dim_FL$ and $x_\alpha$= some $c_i$ for each $\alpha\in\Phi$ , $h_{\alpha}$= some $c_{j}$ with  $\alpha$  some base element  of  $\Phi$  for a Chevalley basis 
\{$c_{i}$\} of  the $C_{l}$ - type  Lie algebra over $\mathbb{C}$ .

\

\

\section{restricted and nonrestricted representations}

\

\

We assume in this section that the ground field $F$ is an algebraically closed field of nonzero characteristic $p.$

\begin{definition}\label{thm3.1}
Let  $(L , [p])$ be a restricted Lie algebra over $F$  and $\chi \in L^{*}$ be a linear form. If a representation
 $\rho_\chi:$$L\longrightarrow$ $ \mathfrak{gl}(V)$   of  $(L, [p])$ satisfies  $\rho_\chi( x^p- x^{[p]})$= $\chi(x)^p id_V$ 
for any $x\in L,$ then $\rho_\chi$  is said  to be a  $\chi- representation.$\newline

 In  this case we say that the representation or the 
corresponding module has a      $p-character  \chi.$ In particular if $\chi$=0 , then $\rho_0$ is  called  a $ restricted$  representation, whereas  $\rho_\chi$ 
for $\chi \neq0$  is called a $nonrestricted$ representation .
\end{definition}

We are well aware that  we have  $\rho_\chi(a)^p -\rho_\chi ( a^{[p]})$=$\chi(a)^p id_V$ for some $\chi\in L^*$,  for any  $a\in L$  and for any irreducible  representation $\rho_\chi.$

\

\section{irreducible nonrestricted representatons of  $C_l-$type Lie algebra }

\

\

In this section we compute the dimension of  some irreducible modules  of  the $C_l-$ type Lie algebra $L$ with a CSA $ H$ over an algebraically closed field $F$ 
of characteristic $p\geq7.$

\begin{prop}\label {thm4.1}

Let  $\alpha$  be any  root  in the root system  $\Phi$ of $L .$ If $\chi(x_\alpha)$ $\neq0,$ then $dim_F$$ \rho_\chi$$(U(L))$ = $p^{2m},$ where $ [Q(U(L)):Q(\mathfrak{Z})]$=$p^{2m}$=$p^{n-l}$ with $\mathfrak{Z}$ the center of $U(L)$  and  $Q$  denotes  the quotient algebra. So the irreducible module corresponding to this representation has $p^m$ as its  dimension.

\end{prop}

\begin{proof}

Let $ \mathfrak {M}_\chi$ be the kernel of this irreducible representation,i.e., a certain (2-sided) maximal ideal of $U(L).$ \newline

(I) Assume first that $\alpha$ is a  short root; then we may put $\alpha$=$\epsilon_1-\epsilon_2$  without loss of generaity since all roots of a given length 
are conjugate under the Weyl group of the root system $\Phi$.\newline

 First we let  $B_i$=$b_{i1}$ $h_{\epsilon_{1}- \epsilon_{ 2}}$+ $b_{i2}$ $h_{\epsilon_{2}-\epsilon_{3}}$+$\cdots$+$b_{i,l-1}$ $h_{\epsilon_{l-1}-\epsilon_{ l }}$ + $b_{il}$ $h_{\epsilon_{l}}$  for $i=  1,2,$ $\cdots,2m,$ where ($b_{i1}$,$b_{i2}$   $\cdots,$$b_{il}$) $\in F^{l}$ are chosen so that any $ (l+1)-$$ B_i$'s  are linearly independent in $\mathbb{P}^l(F),$ the  $B$   below  becomes an $F-$ linearly independent  set in $U(L)$ if necessary and $x_\alpha$$B_i$ $\not\equiv$$B_i$$x_\alpha$ for $\alpha$=$\epsilon_1-$$\epsilon_2.$ \newline

In  $U(L)$/$\mathfrak{M}_\chi$ we claim that we have a basis 
$B$:= 
$\{(B_1 +A_{\epsilon_{1}-\epsilon_ { 2}})^{i_1}\otimes(B_2 +A_{-(\epsilon_{1}-\epsilon_ {2})})^{i_2}\otimes \cdots \otimes(B_{2l-2} +A_{-(\epsilon_{l-1}-\epsilon_{l})})^{i_{2l-2}}\otimes(B_{2l-1}+ A_{2\epsilon_{l}})^{i_{2l-1}}\otimes(B_{2l}+A_{-2\epsilon_{l}})^{i_{2l}}\otimes(\otimes_{j=2l+1}^{2m}(B_j+A_{\alpha_{j}})^{i_{j}}) ; 0 \leq i_{j}\leq p-1 \}$, \newline
where we put\newline
$A_{\epsilon_{1}-\epsilon_{2}}$= $x_\alpha$=
$x_{\epsilon_{1}-\epsilon_{2}},$ \newline
$A_{\epsilon_{2}-\epsilon_{1}}$=$c_{-(\epsilon_{1}-\epsilon_{2})}$ +$(h_{\epsilon_{1}-\epsilon_{2}}+1)^2$ +4$x_\alpha$ $x_{-\alpha},$ \newline
$A_{{\epsilon_{2}}{\pm}\epsilon_{3}}$=$x_{\pm2\epsilon_{3}}$ $(c_{\epsilon_{2}\pm\epsilon_{3}}+ x_{\epsilon_{2}\pm\epsilon_{3}}x_{-(\epsilon_{2}\pm\epsilon_{3})}\pm x_{\epsilon_{1}\pm\epsilon_{3}}x_{-(\epsilon_{1}\pm\epsilon_{3})}),$ \newline
$A_{\epsilon_{1}+\epsilon_{2}}$=$x_{\epsilon_{1}-\epsilon_{2}}^2$ $(c_{\epsilon_{1}+\epsilon_{2}}+3x_{\epsilon_{1}+\epsilon_{2}}x_{-\epsilon_{1}-\epsilon_{2}}\pm 2x_{2\epsilon_{1}}x_{-2\epsilon_{1}}\pm 2 x_{2\epsilon_{2}}x_{-2\epsilon_{2}}),$ $A_{\epsilon_{2}\pm\epsilon_{k}}$=$x_{\epsilon_{3}\pm\epsilon_{k}}( c_{\epsilon_{2}\pm\epsilon_{k}}+x_{\epsilon_{2}\pm\epsilon_{k}}x_{-(\epsilon_{2}\pm\epsilon_{k})} \pm x_{\epsilon_{1}\pm\epsilon_{k}}x_{-(\epsilon_{1}\pm\epsilon_{k})} ),$ \newline
$A_{2\epsilon_{2}}$=$x_{2\epsilon_{3}}^2 (c_{2\epsilon_{2}}+2x_{2\epsilon_{2}}x_{-2\epsilon_{2}}\pm 3x_{\epsilon_{1}+ \epsilon_{2}}x_{-\epsilon_{1}-\epsilon_{2}} +2x_{2\epsilon_{1}}x_{-2\epsilon_{1}}),$ \newline
$A_{-2\epsilon_{1}}$=
$x_{-2\epsilon_{3}}^2 ( c_{-2\epsilon_{1}}+ 2x_{-2\epsilon_{1}}x_{2\epsilon_{1}}\pm3x_{-\epsilon_ {1}-\epsilon_{2}}x_{\epsilon_{1}+\epsilon_{2}}
\pm2x_{-2\epsilon_{2}}x_{2\epsilon_{2}}), $ $A_{-(\epsilon_{1}\pm\epsilon_{3})}$= $x_{-(\pm\epsilon_{3})}(c_{-(\epsilon_{2}\pm\epsilon_{3})}
+  x_{\epsilon_{2}\pm\epsilon_{3}}x_{-(\epsilon_{2}\pm\epsilon_{3})}\pm x_{\epsilon_{1}\pm\epsilon_{3}}x_{-(\epsilon_{1}\pm\epsilon_{3})}),$
$A_{-(\epsilon_{1}\pm\epsilon_{k})}$= $x_{-(\epsilon_{3}\pm\epsilon_{k})}(c_{-(\epsilon_{1}\pm\epsilon_{k})}+ x_{\epsilon_{2}\pm\epsilon_{k}}x_{-(\epsilon_{2}\pm\epsilon_{k})}\pm x_{\epsilon_{1}\pm\epsilon_{k}}x_{-(\epsilon_{1}\pm\epsilon_{k})} ),$ $A_{2\epsilon_{l}}$= $x_{2\epsilon_{l}}^2, $ \newline
$A_{-2\epsilon_{l}}$= $x_{-2\epsilon_{l}}^2, $ 
\newline

with the sign chosen so that they commute with $x_{\alpha}$ and with $c_{\alpha}\in F$ 
chosen so that $A_{\epsilon_{2}-\epsilon_{1}}$ and parentheses are invertible.  For any other root $ \beta$ we put $A_{\beta}$= $x_{\beta}^2 $ or $x_{\beta}^3 $ if possible. Otherwise attach to these sorts the parentheses(        ) used for designating $A_{-\beta}$ so that  $A_\gamma  \forall \gamma \in \Phi$ may commute with $x_\alpha$.\newline

We shall prove that $B$ is a basis in $U(L))$/$\mathfrak {M}_\chi$. By virtue of P-B-W theorem, it is not difficult to see that $B$ is evidently a linearly independent set  over $F$ in $U(L)$. Furthermore $\forall$ $\beta$ $ \in\Phi$, $A_{\beta}\notin\mathfrak {M}_\chi$(see detailed proof below).\newline

We shall prove that a nontrivial linearly dependence equation leads to absurdity.We assume first that there is a dependence equation which is of least degree with respect to $h_{\alpha_{j}}\in H$ and the number of whose highest degree terms is also least. \newline

In case it is conjugated by $x_{\alpha}$,then there arises a nontrivial dependence equation of lower degree than the given one, which contradicts our assumption.
  Otherwise it reduces to one of the following forms:\newline

(i) $x_{2\epsilon_{j}}$$K$ + $K'$ $\in$ $\mathfrak{M}_\chi$ ,

(ii) $x_{-2\epsilon_{j}}$$K$+ $K'$ $\in$ $\mathfrak{M}_\chi$ ,

(iii)$x_{\epsilon_{j}+\epsilon_{k}}$ $K$ + $K'$$\in$ $\mathfrak{M}_\chi$,

(iv)$x_{-\epsilon{j}-\epsilon_{k}}$$K$ + $K'$ $\in$ $\mathfrak{M}_\chi$,

(v)$x_{\epsilon_{j}-\epsilon_{k}}$$K$ + $K'$ $\in$ $\mathfrak{M}_\chi$ ,

where $K$, $K'$ commute with $x_{\alpha}$.\newline

For the case (i), we deduce successively 
$x_{\epsilon_{2}-\epsilon_{j}} x_{2\epsilon{j}}$$K$ + $x_{\epsilon_{2}-\epsilon{j}}$$K'$ $\in$ $\mathfrak{M}_\chi$

$\Rightarrow$ $x_{\epsilon_{2}+\epsilon_{j}}$$K$ + $x_{2\epsilon_{j}}$ $x_{\epsilon_{2}-\epsilon_{j}}$$K$ + $x_{\epsilon_{2}-\epsilon{j}}$$K'$
 $\in$ $\mathfrak{M}_\chi$ $\Rightarrow$($x_{\epsilon_{1}+\epsilon_{j}}$ or $x_{2\epsilon_{1}}$)$K$ + $x_{2\epsilon_{j}}$($x_{\epsilon_{1}-\epsilon_{j}}$ or $h_{\epsilon_{1}-\epsilon_{2}}$)$K$ +  ($x_{\epsilon_{1}-\epsilon_{j}}$ or $h_{\epsilon_{1}-\epsilon_{2}}$)$K'$ $\in$ $\mathfrak{M}_\chi$ by $adx_{\epsilon_{1}-\epsilon_{2}}$

if $j$$\neq{1}$ or $j$=1 respectively, so that by successive $ adx_{\alpha}$ and rearrangement we get $x_{\epsilon_{1}\pm\epsilon_{j}}$$K+ K''$  $\in$ $\mathfrak{M}_\chi$ for some $K''$ commuting with $x_{\alpha}$ in view of  the start equation. So (i) reduces to (iii),(iv) or (v). \newline
Similarly as in (i)  and by adjoint operations , (ii) reduces to (iii),(iv) or (v). Also (iii),(iv) reduces to the form (v) putting $\epsilon_{j}$= -(-$\epsilon_{j}$), $\epsilon_{k}$= -(-$\epsilon_{k})$. \newline

Hence we have only to consider the case (v).
We consider $x_{\epsilon_{k}-\epsilon_{2}}$ $x_{\epsilon_{j}-\epsilon_{k}}$ $K$+ $x_{\epsilon_{k}- \epsilon_{2}}$$K'$ $\in$ $\mathfrak{M}_\chi$ ,
so that ($x_{\epsilon_{j}-\epsilon_{2}}$+ $x_{\epsilon_{j}-\epsilon_{k}}$$x_{\epsilon_{k}-\epsilon_{2}}$)$K$ + $x_{\epsilon_{k}-\epsilon_{2}}$$K'$

$\in$ $\mathfrak{M}_\chi$ for $j,k$$\neq$1,2 . We thus have $x_{\epsilon_{j}-\epsilon_{2}}$$K$ + ($x_{\epsilon_{j}-\epsilon_{k}}$$x_{\epsilon_{k}-\epsilon_{2}}$ $K$ + $x_{\epsilon_{k}-\epsilon_{2}}$$K'$) $\in$ $\mathfrak{M}_\chi$, so that we may put this last (       )= another $K'$ alike as in the equation   (v).\newline

 Hence we need to show that  $x_{\epsilon_{j}-\epsilon_{2}}$$K$ + $K'$ $\in$ $\mathfrak{M}_\chi$ leads to absurdity.
We consider $x_{\epsilon_{2}-\epsilon_{j}}$$x_{\epsilon_{j}-\epsilon_{2}}$$K$ + $x_{\epsilon_{2}-\epsilon_{j}}$$K'$ $\in$ $\mathfrak{M}_\chi$
$\Rightarrow$ ($h_{\epsilon_{2}-\epsilon_{j}}+ x_{\epsilon_{j}-\epsilon_{2}}x_{\epsilon_{2}-\epsilon_{j}})K+  x_{\epsilon_{2}-\epsilon_{j}}K' \in \mathfrak{M}_\chi$ $\Rightarrow $    ($x_{\epsilon_{1}-\epsilon_{2}}$$\pm$$x_{\epsilon_{j}-\epsilon_{2}}$$x_{\epsilon_{1}-\epsilon_{j}}$)$K$ + $x_{\epsilon_{1}-\epsilon_{j}}$ $K'$$\in$ $\mathfrak{M}_\chi$  by $adx_{\epsilon_{1}-\epsilon_{2}}$ $\Rightarrow$ either $x_{\epsilon_{1}-\epsilon_{2}}$$K$ $\in$ $ \mathfrak{M}_\chi$ or  ( $x_{\epsilon_{1}-\epsilon_{2}}$ + $x_{\epsilon_{j}-\epsilon_{2}}$$x_{\epsilon_{1}-\epsilon_{j}}$)$K$+ $x_{\epsilon_{1}-\epsilon_{j}}$$K'$ $\in$ $\mathfrak{M}_\chi$ depending on [$x_{\epsilon_{j}-\epsilon_{2}}$, $x_{\epsilon_{1}-\epsilon_{j}}$]= +$x_{\epsilon_{1}-\epsilon_{2}}$  or    -$x_{\epsilon_{1}-\epsilon_{2}}$. The former case leads to $K$ $\in$ $\mathfrak{M}_\chi$, a contradiction. \newline

For the latter case

we consider \newline
$x_{\epsilon_{1}-\epsilon_{2}}$$K$ + ( $x_{\epsilon_{j}-\epsilon_{2}}$$x_{\epsilon_{1}-\epsilon_{j}}K$ + $x_{\epsilon_{1}-\epsilon_{j}}$$K'$)

$\in$ $\mathfrak{M}_\chi$.\newline

So we may put\newline

   $(\ast) x_{\epsilon_{1}-\epsilon_{2}}$$K$ + $K''$ $\in$$ \mathfrak{M}_\chi$,
where $K''$=$x_{\epsilon_{j}-\epsilon_{2}}$$x_{\epsilon_{1}-\epsilon_{j}}$$K$+ $x_{\epsilon_{1}-\epsilon_{j}}$$K'$.
Thus $x_{\epsilon_{2}-\epsilon_{1}}$$x_{\epsilon_{1}-\epsilon_{2}}$$K$ + $x_{\epsilon_{2}-\epsilon_{1}}$$K''$ $\in$ $\mathfrak{M}_\chi$.
From $w_{\epsilon_1- \epsilon_2}:=(h_{\epsilon_1- \epsilon_2}+ 1)^2 +4 x_{\epsilon_2- \epsilon_1}x_{\epsilon_1- \epsilon_2}\in $ the center of  $U(\frak{sl}_2(F)), $ we get $ 4^{-1}\{w_{\epsilon_1- \epsilon_2}- (h+ 1)^2\}K+ x_{\epsilon_2- \epsilon_1}K'' \equiv 0 $  modulo $\frak M_\chi$.\newline
If $x_{\epsilon_2- \epsilon_1}^p\equiv c $ which is a constant,then \newline

$(\ast \ast)4^{-1}x_{\epsilon_2-\epsilon_1}^{p-1}\{w_{\epsilon_1- \epsilon_2}- (h_{\epsilon_1- \epsilon_2}+ 1)^2\}K+ cK''\equiv 0 $ is obtained.
\newline
From $(\ast),(\ast \ast)$, we have $4^{-1}x_{\epsilon_2- \epsilon_1}^{p-1}\{w_{\epsilon_1- \epsilon_2}- (h_{\epsilon_1- \epsilon_2}+ 1)^2\}K- cx_{\epsilon_1- \epsilon_2}K\newline
\equiv 0.$
Multiplying $x_{\epsilon_1- \epsilon_2}^{p-1}$ to this equation,we obtain \newline

 $(\ast \ast \ast)4^{-1}x_{\epsilon_1- \epsilon_2}^{p-1}x_{\epsilon_2- \epsilon_1}^{p-1}\{w_{\epsilon_1- \epsilon_2}- (h_{\epsilon_1- \epsilon_2}+ 1)^2\}K- cx_{\epsilon_1- \epsilon_2}^pK\equiv 0.  $
By making use of $w_{\epsilon_1- \epsilon_2}$, we may have from $(\ast \ast \ast)$ an equation of the form \newline
( a polynomial of degree $\geq 1$ with respect to  $ h_{\epsilon_1- \epsilon_2})K- cx_{\epsilon_1- \epsilon_2}^pK\equiv 0.   $\newline

Finally if we use conjugation and subtraction,then we are led to a  contradiction $K\in \frak M_\chi.$\newline

(II)Assume next that $\alpha$ is a long root; then we may put $\alpha=2 \epsilon_{1}$ because all roots of the same length are conjugate under the Weyl group of $\Phi$ . Similarly as in (I), we let  $B_{i}$:= the same as in (I) except that this time  $\alpha=2\epsilon_{1}$ instead of $\epsilon_{1}-\epsilon_{2}$ . 
\newline

We claim that we have a basis $B$:= $\{(B_{1}+ A_{2\epsilon_{1}})^{i_{1}}\otimes(B_{2} + A_{-2\epsilon_{1}})^{i_{2}}\otimes(B_{3}+ A_{\epsilon_{1}-\epsilon_{2}})^{i_{3}}\otimes(B_{4}+A_{-(\epsilon_{1}-\epsilon_{2})})^{i_{4}}\otimes\cdots \otimes(B_{2l}+ A_{-(\epsilon_{l-1}-\epsilon_{l})})^{i_{2l}}\otimes(B_{2l+1}+ A_{2\epsilon_{l}})^{i_{2l+ 1}}\otimes(B_{2l+ 2}+ A_{-2\epsilon_{l}})^{i_{2l+ 2}}\otimes(\otimes_{j=2l+ 3}^{2m}(B_{j}+ A_{\alpha_{j}})^{i_{j}}; 0 \leq i_{j}\leq p-1\}$ , \newline

where we put \newline
$A_{2\epsilon_{1}}$= $x_{2\epsilon_{1}}$, \newline
$A_{-2\epsilon_{1}}$= $c_{-2\epsilon_{1}}$+ $(h_{2\epsilon_{1}}+ 1)^{2}+ 4x_{-2\epsilon_{1}}x_{2\epsilon_{1}} $, \newline
$A_{-\epsilon_{1}\pm\epsilon_{2}}$=$ x_{-\epsilon_{3}\pm\epsilon_{2}}$$( c_{-\epsilon_{1}\pm\epsilon_{2}}$ $\pm$$x_{-\epsilon_{1}\pm\epsilon_{2}}$$x_{\epsilon_{1}\mp\epsilon_{2}}$$\pm$$x_{\epsilon_{1}\pm\epsilon_{2}}$$x_{-\epsilon_{1}\mp\epsilon_{2}})  $,

 $A_{-\epsilon_{!}\pm \epsilon_ {j}}= x_{-\epsilon_{2}\pm\epsilon_{j}}
( c_{-\epsilon_{1}\pm\epsilon_{j}} 
+ x_{\pm{\epsilon_{j}}-\epsilon_{1}}x_{\epsilon_{1}\mp\epsilon_{j}}$
$\pm$$x_{\epsilon_{1}\pm\epsilon_{j}}x_{-\epsilon_{1}\mp\epsilon_{j}})$ , and for any other root $\beta$ we put  $A_{\beta}= x_{\beta}^2 $ or $x_{\beta}^3 $  if possible. Otherwise attach to these the parentheses (           ) used for designating $A_{-\beta}$. 
 Likewise as in case (I),  we shall prove that $B$ is a basis in $U(L)$/$\mathfrak{M}_\chi$. By virtue of P-B-W theorem, it is not difficult to see that $B$ is evidently a linearly independent set over $F$ in $U(L)$. Moreover $\forall\beta$$\in$$ \Phi$, $A_{\beta}\notin \mathfrak{M}_\chi$(see detailed proof below).\newline

We shall prove that a nontrivial linearly dependence equation leads to absurdity. We assume first that there is a dependence equation which is of least degree with respect to $h_{\alpha_{j}}$ $\in$$H$ and the number of  whose highest degree terms is also least. If it is conjugated by  $x_{\alpha}$, then there arises a nontrivial dependence equation of least degree than the given one,which contravenes our assumption. \newline

 Otherwise it reduces to one of the following forms: \newline
(i) $x_{2\epsilon_{j}}K + K'\in \mathfrak{M}_\chi$ ,\newline
(ii)  $x_{-2\epsilon_{j}}K + K' \in \mathfrak{M}_\chi$,\newline
(iii)$ x_{\epsilon_{j}+ \epsilon_{k}}K+ K' \in \mathfrak{M}_\chi$ ,\newline
(iv) $ x_{-\epsilon_{j}-\epsilon_{k}}K+ K' \in \mathfrak{M}_\chi$,\newline
(v) $x_{\epsilon_{j}-\epsilon_{k}}K + K' \in \mathfrak{M}_\chi$ ,\newline
where $K$ and $K'$ commute with $x_{\alpha}= x_{2\epsilon_{1}}$.\newline

For the case (i) , we consider  a particular case  $j$=1 first; if we assume $x_{2\epsilon_{1}}K+ K' \in \mathfrak{M}_\chi$, then we are led to a contradiction according to the similar argument ($\ast$) as in (I). So we assume  $x_{2\epsilon_{j}}K + K' \in \mathfrak{M}_\chi$ with $j\geq2$. Now we have  
$x_{2\epsilon_{j}}K+ K' \in \mathfrak{M}_\chi$ $\Rightarrow$ $x_{-\epsilon_{1}-\epsilon_{j}}x_{2\epsilon_{j}}K + x_{-\epsilon_{1}-\epsilon_{j}}K'\in 
\mathfrak{M}_\chi \Rightarrow $ $x_{-\epsilon_{1}+ \epsilon_{j}}K + x_{2\epsilon_{j}}x_{-\epsilon_{1}-\epsilon_{j}}K + x_{-\epsilon_{1}-\epsilon_{j}}K' \in \mathfrak{M}_\chi \Rightarrow$ by $adx_{2\epsilon_{1}},   x_{\epsilon_{1}+ \epsilon_{j}}K + x_{2\epsilon_{j}}x_{\epsilon_{1}-\epsilon_{j}}K + x_{\epsilon_{1}-\epsilon_{j}}K' \in \mathfrak{M}_\chi$ is obtained. Hence (i) reduces to (iii).\newline

 Similarly (ii)reduces to (iii) or (iv) or (v). So we have only to consider (iii), (iv) , (v). However (iii), (iv), (v) reduce to $x_{2\epsilon_{1}}K + K'' \in \mathfrak{M}_\chi$ after all considering the situation as in (I). Similarly following the argument as in (I), we are led to a contradiction $K \in \mathfrak{M}_\chi$ .

\end{proof}

\begin{cor}\label{thm4.2} The Weyl module $W_{\chi}(L)$= the Verma module $V_{\chi}(L)$, where $ \chi$ is a nonzero character of $L$ over an algebraically 
closed field $F$ of characteristic $p\geq7.$ In other words the dimension of any irreducible $L$-module over $F$ is  $p^{m}$ if the irreducible module is associated with 
a nonzero character $\chi$$\neq{0}$.

\end{cor}

\begin{proof}
It is obvious by the proposition(4.1).

\end{proof}

So we might as well extend the argument in the proof of the proposition(4.1) to that of other type simple Lie algebras.
\begin{rem}
James E. Humphreys indicated in [3] that there is an  $sp_{4}(F)-$ irreducible module of dimension $p^{3}$ in the above situation which has nearly nothing to do with characteristic $p$. \newline

However according to the argument  similar as  our argument  in the above propositions, the dimension of $sp_{4}(F)-$irreducible module is equal to $p^{4}$ in our situation because $l=2,m=4$, and $ n=10$.
Our computation so far  makes us to conjecture  results alike for other classical modular simple Lie algebras.

\end{rem}

\
\

\bibliographystyle{amsalpha}

\end{document}